\documentclass[12pt]{article}
\usepackage{listings}
\usepackage{xcolor}
\usepackage[utf8]{inputenc}
\usepackage{amsmath, amssymb}
\usepackage{amsthm}
\usepackage{geometry}
\usepackage{graphicx}
\usepackage{verbatim}
\usepackage{ulem}
\usepackage{orcidlink}
\usepackage{ulem}
\usepackage{subcaption}
\usepackage{authblk}
\usepackage{hyperref}

\title{On the closed solution of a  problem coupling  fluid  infiltration with a hydration reaction}

\author{Diego Guevara$^{1,2}$, Sabrina Roscani$^{1,3}$, Piotr Rybka$^{4}$, and Vaughan Voller$^{5}$}

\date{
  \small
  $^1$CONICET, Argentina\\
  $^2$Departamento de Matemática, ECEN, FCEIA, Universidad Nacional de Rosario, Pellegrini 250, Rosario, 2000, Argentina\\
  $^3$Departamento de Matemática, FCE, Universidad Austral, Paraguay 1950, Rosario, 2000, Argentina\\
  $^4$Faculty of Mathematics, Informatics and Mechanics, University of Warsaw, ul. Banacha 2 02-097, Warsaw , Poland\\
  $^5$Department of Civil, Environmental and Geo- Engineering, University of Minnesota, Minneapolis, MN 55455, USA\\
\\
  (dguevara@fceia.unr.edu.ar; sroscani@austral.edu.ar ; rybka@mimuw.edu.pl ;
volle001@umn.edu)
}

\definecolor{wineRed}{rgb}{0.7,0,0.3}

\newtheorem{theo}{Theorem}

\newtheorem{remark}{Remark}

\newcommand{\bbR}{{\mathbb R}}

\author{Diego Guevara$^{1,2}$, Sabrina Roscani$^{1,3}$, Piotr Rybka$^{4}$, and Vaughan Voller$^{5}$}

\date{
  \small
  $^1$CONICET, Argentina\\
  $^2$Departamento de Matemática, ECEN, FCEIA, Universidad Nacional de Rosario, Pellegrini 250, Rosario, 2000, Argentina\\
  $^3$Departamento de Matemática, FCE, Universidad Austral, Paraguay 1950, Rosario, 2000, Argentina\\
  $^4$Faculty of Mathematics, Informatics and Mechanics, University of Warsaw, ul. Banacha 2 02-097, Warsaw , Poland\\
  $^5$Department of Civil, Environmental and Geo- Engineering, University of Minnesota, Minneapolis, MN 55455, USA
  (dguevara@fceia.unr.edu.ar; sroscani@austral.edu.ar ; rybka@mimuw.edu.pl ;
volle001@umn.edu)
}
\begin{document}
\maketitle

\begin{abstract}
We present a one-dimensional model for water infiltration coupled with a hydration reaction, relevant to coupled transport and chemical processes in the Earth's subsurface. In this model the sharp interface separating the saturated and dry regions evolves over time, leading to a moving free boundary problem. Consistent with recently presented numerical treatments in the literature, this solution indicates an interesting dynamic for the free boundary. At early time, for a given domain  porosity,  the infiltration front advances with a specific square-root-in-time behavior.  At later time, depending on the consumption of the hydration reaction, the front advance  can exhibit square-root, linear, or exponential forms.   
The presented closed solution provides  an analytical tool that can be used to quantify important behavior in coupled transport and reaction systems.
\end{abstract}

\section{Motivation and formulation of infiltration--hydration model }
Environmental remediation processes, e.g., the storage of $CO_2$ through mineralization, \cite{Matter2016}, \cite{nisbet2024carbon}, often involve the infiltration of a reactive fluid into the Earth’s subsurface \cite{SteefelMaher2009}.  An important and illustrative example of such a process is the hydration reaction that occurs when water infiltrates into an initially dry porous rock mass. A process that under the right conditions, can lead to so called “reaction-driven fractures”, \cite{uno2022volatile}, that enhance the transport of the reactive fluids.  Under the condition of a fixed pressure head applied at the origin of a one-dimensional horizontal domain and the assumptions of (i) a homogeneous porous media,  (ii) a sharp infiltration front separating wet and dry regions (the Green-Ampt assumption \cite{green1911studies, neuman1976wetting}), and (iii) a  first order hydration reaction that consumes the infiltrating water, a model of this process can be constructed, \cite{detournay2025coupled, voller2025assessment}.  Within the saturated region $0<x<s(t)$, the flow satisfies Darcy's law,
\begin{equation*}
q=-\frac{\kappa}{\mu}\frac{\partial p}{\partial x},
\end{equation*}
where $q$ is the Darcy flux, $\kappa$ the intrinsic permeability, $\mu$ the dynamic viscosity, and $p=p(x,t)$ the fluid pressure at position $x$ and time $t$.
If $C=C(x,t)$ denotes the concentration of the reactive species the first order reaction evolution is governed by
\begin{equation*}
\frac{\partial C}{\partial t}
(x,t)=
-kC(x,t),
\qquad
x>0,
\qquad
t>h(x),
\end{equation*}
with
\begin{equation*}
C(x,h(x))=C_0,
\end{equation*}
where $k$ is the reaction rate constant and $h(x)$ denotes the arrival time of the infiltration front at position $x$.
Combining Darcy's law with conservation of mass yields
\begin{equation*}
\frac{\kappa}{\mu}\frac{\partial^2 p}{\partial x^2}
=
\nu k C,
\qquad
0<x<s(t),
\end{equation*}
where $\nu$ is the molar volume of the hydration reactant. Assuming a unit stoichometry, the right hand side represents the volume of  water per unit volume media absorbed by the hydration reaction. 
The pressure boundary conditions are
\begin{equation*}
p(0,t)=p_0,
\qquad
p(s(t),t)=0, \quad t>0
\end{equation*}
while conservation of mass at the moving infiltration front---the free boundary---gives the Stefan-type condition
\begin{equation*}
-\frac{\kappa}{\mu}
\left.
\frac{\partial p}{\partial x}
\right|_{s}
=
\phi\frac{ds}{dt}.
\end{equation*}
where $\phi>0$ is the porosity of the media.

Recent works \cite{detournay2025coupled, voller2025assessment}  numerically investigates the above infiltration-hydration  model, obtaining insight on the conditions that will lead to reaction fracture. The numerical predictions also exhibit interesting early and late time behaviors. At very early times, while the hydration reaction is weak, the infiltration front advances with the square root in time behavior associated with the problem of infiltration in the absence of hydration. As time advance and the hydration reaction becomes more extensive  the movement of the front drifts away from the square root in time dependence.  At later times, however, the water consumption of the hydration reaction concentrates in the close vicinity of infiltration front and a square root in time dependence, with a smaller pre-fector to the early time, is recovered.  Here we obtain a closed solution for this infiltration model with hydration. A solution that provides an analytical tool for identifying the conditions for reaction fracture and analytically recovers the cross over form the late and early square root in time infiltration behavior predicted by the numerical solution \cite{detournay2025coupled, voller2025assessment} 

\section{Dimensionless form of infiltration--hydration model }

Introducing the length and time scales  $\ell_*$
  and $t^*$, we can define the following dimensionless variables
\[
\overline{x} = \frac{x}{\ell_*}, \qquad
\overline{t} = \frac{t}{ t^*}, \qquad
\overline{s} = \frac{s}{\ell_*}, \qquad
\overline{p} = \frac{p}{p_0}, \qquad
\overline{\Gamma} = 1 - \frac{C}{C_0},
\]
where the length and time scales are $
\ell_* = \sqrt{2 \frac{\kappa}{\mu} \frac{p_0}{k}}$ and $
t_* = \frac{1}{k}$.

After that
we obtain the dimensionless quasi-steady problem for 
finding:

\begin{enumerate}
\item The position of the free boundary $\overline{s}\in C^1( \bbR_0^+) $ and
\item  the scalar field functions $\overline{p}\in C(Q)$, $\overline{\Gamma}\in L^\infty(Q)$, where
$
Q=\{(\overline x,\overline t)\colon 0\leq \overline x\leq \overline{s}(\overline t),\, \overline t>0  \}
$ 
and $\overline{p}(\cdot,\overline t)\in C^2(0,\overline{s}(\bar{t}
))$ for all $\overline t>0$,  $\overline{\Gamma}(\bar{x}
, \cdot)\in C^1(\bar{h}(\bar{x}),\infty)$ for all  $\overline x$ in $(0,\overline{s}(\bar{t}
))$.
\end{enumerate}
The triple $(\overline{s}, \overline{p}, \overline{\Gamma})$
is 
such that
\begin{equation}\label{dim-prob}
\begin{array}{lll}
(i)& \frac{\partial^2 \overline{p}}{\partial \overline{x}^2} = 2Y(1-\overline{\Gamma}), & 0<\overline{x}<\overline{s}(\overline{t}), \; \overline{t}>0,  \\ 
(ii) & -\frac{\partial \overline{p}}{\partial \overline{x}}(\overline{s}(\overline{t}),\overline{t}) = 2\phi \frac{d\overline{s}}{d\overline{t}},  & \overline{t}>0,\\
(iii)& \frac{\partial \overline{\Gamma}}{\partial \overline{t}} = (1-\overline{\Gamma}), & 0<\overline{x}<\overline{s}(\overline{t}), \; \overline{t}>0, \\ 
(iv) & \overline{p}(0,\overline{t}) = 1,\, \overline{p}(\overline{s}(\overline{t}),\overline{t}) = 0, & \overline{t}>0,\\
(v) & \overline{\Gamma}(\overline{s}(\overline{t}),\overline{t}) = 0, & \overline{t}>0,\\
(vi)& \overline{s}(0)= 0. & 
\end{array}
\end{equation}
Here, $Y =\nu C_0>0$ is the total volume of water consumed by the hydration reaction per unit volume of the porous  media.

\section[An explicit solution for the free boundary s]{An explicit solution for the free boundary $\overline s$}
 Our goal is to establish existence of solutions to \eqref{dim-prob}. Here is our observation.
\begin{theo}\label{t1} Let  $\phi>0$ and $Y\in \bbR$, then there exists a unique  global in time solution to \eqref{dim-prob}. Moreover, the solution $(\overline{s},  \overline{\Gamma}, \overline{p})$ is given by  formulas \eqref{s-sol},  \eqref{G-sol} and \eqref{p-sol}, and the free boundary $\bar{s}$ is a strictly increasing function.
\end{theo}
\begin{proof}
In the aim to solve problem \eqref{dim-prob}, we integrate $(\ref{dim-prob})-(i)$ between $\overline{x}$ and $\overline{s}(\overline{t})$,  
\begin{equation}\label{class-2}
    \frac{ \partial \overline{p}}{\partial \overline{x}}(\overline{x},\overline{t})=-2\phi \overline{s}'(\overline{t})-2Y\int_{\overline{x}}^{\overline{s}(\overline{t})}(1-\overline{\Gamma}(\overline{y},\overline{t}))\,d\overline{y}.
\end{equation}
Integrating now \eqref{class-2} between $0$ and $\overline{s}(\overline{t})$ using Fubini's theorem and condition $(\ref{dim-prob})-(iv)$, we obtain
\begin{equation}\label{class-4}
    2\phi \overline{s}'(\overline{t})\overline{s}(\overline{t})=1-2Y\int_0^{\overline{s}(\overline{t})}\overline{y}(1-\overline{\Gamma}(\overline{y},\overline{t}))\,d\overline{y}.
\end{equation}

Now, we define the auxiliary function 
\begin{equation}\label{J}
J(\overline{t}) := \int_0^{\overline{s}(\overline{t})} \overline{y} \bigl(1 - \overline{\Gamma}(\overline{y},\overline{t})\bigr) \, d\overline{y}.  
\end{equation}
Notice that $J$ plays a fundamental role in finding the exact solution through forming an appropriate ODE in $J$. Towards this purpose, we differentiate it with respect to time and use equation $(\ref{dim-prob})-(iii)$ and condition $(\ref{dim-prob})-(v)$ to obtain
\begin{equation}\label{class-5}
J'(\overline{t}) + J(\overline{t}) = \overline{s}(\overline{t})\overline{s}'(\overline{t}). 
\end{equation}
Then, due to \eqref{class-4} and \eqref{class-5} we obtain the desired ODE for the auxiliary function \eqref{class-4}

\begin{equation}\label{class-6}
J'(\overline{t}) + \left(1 + \frac{Y}{\phi}\right) J(\overline{t}) = \frac{1}{2\phi}.    
\end{equation}
We distinguish three cases according to the values of \(Y+\phi\) and \(Y\).
\\ \textit{Case 1.} \(Y=0\).
In this case \eqref{class-4}, with the initial condition \(\bar s(0)\), gives  
\begin{equation}\label{Y=0}
\bar s(t)= \sqrt{\frac {\bar t}{\phi}},
\end{equation}
the standard  square root in time solution for the movement of an infiltration front in the absences of a hydration reaction. 
\\ \textit{Case 2.} \(Y+\phi=0\) and \(Y \ne 0\).
Since \(Y+\phi=0\), we have \(1+\frac{Y}{\phi}=0\), and \eqref{class-6} reduces to
$
J'(\overline{t})=\frac{1}{2\phi}.
$
Using the initial condition \(J(0)=0\), which follows from \eqref{dim-prob}\textit{-(vi)}, we obtain
$J(\overline{t})=\frac{\overline{t}}{2\phi},$  from where
\begin{equation}\label{Y-PHI=0}
\bar s^2(\bar t) =\frac{1}{\phi} \bar t + 
\frac  {1}{2\phi} \bar t^2.
\end{equation}\\  
\textit{Case 3.} \(Y+\phi\neq0\) and \(Y \ne 0\).

Solving \eqref{class-6} and noting that  the initial condition $\ref{dim-prob}-(vi)$ gives  \(J(0)=0\)  we obtain
\begin{equation}\label{sol-J}
    J(\overline{t}) = \frac{1}{2\phi + 2Y} \left[1-\exp\left(-\left(1 + \frac{Y}{\phi}\right)\overline{t}\right)\right].
\end{equation}
Next, we insert \eqref{sol-J} in \eqref{class-4} and integrate between $0$ and $\overline{t}$ to obtain an explicit representation for $\overline{s}^2$
\begin{equation}\label{s^2}
    \overline{s}(\overline{t})^2 = \frac{1}{\phi + Y} \overline{t} + \frac{Y}{(\phi + Y)^2} \left(1 - \exp\left(-\left(1 + \frac{Y}{\phi}\right)\overline{t}\right)\right).
\end{equation}

It remains to verify that $\overline{s}$ is well defined.  If $Y+\phi > 0$, this follows directly from the assumption $\phi>0$. If $Y+\phi < 0$ then  $\frac{Y}{\phi}<-1$, using the inequality $e^x-1>x$, which holds in $\mathbb{R}^+$, we get the desired result. Hence, $ \bar{s}$ is well defined for all $\bar{t}\ge 0$.

Moreover, in both cases, \eqref{sol-J} and \eqref{class-6} imply $J(\bar{t})>0$ and $J'(\bar{t})>0$ for $\bar{t}>0$; together with \eqref{class-5}, this gives $\bar{s}(t) \bar{s}'(\bar{t})>0$.

Finally,  to respect of the condition that the domain is in the positive half-space, we take the positive square root to obtain the exact formula for the free boundary
\begin{equation}\label{s-sol}
    \overline{s}(\overline{t}) = \sqrt{ \frac{1}{\phi + Y} \overline{t} + \frac{Y}{(\phi + Y)^2} \left(1 - \exp\left(-\left(1 + \frac{Y}{\phi}\right)\overline{t}\right)\right) }.
\end{equation}

The exact solution \eqref{s-sol} describes the advance of the free boundary, whose behavior had been studied numerically in previous work
 \cite{detournay2025coupled, voller2025assessment}
and whose particular features will be examined in the next section.
Now, since \eqref{s-sol} is a strictly increasing function, its inverse $\overline{s}^{-1}$ is well defined and  the boundary condition $(\ref{dim-prob})-(v)$ becomes an initial condition via 
\begin{equation}\label{IC-G}
  \overline{\Gamma}(\overline{x}, \overline{s}^{-1}(\overline{x})) = 0,  
\end{equation}
leading to the closed  implicit  representation for the hydration function
\begin{equation}\label{G-sol}
   \overline{\Gamma}(\overline{x},\overline{t}) = 1 - \exp\left( \overline{t} - \overline{s}^{-1}(\overline{x}) \right).
\end{equation}
Having obtained $\overline{\Gamma}$ in a closed form, the pressure $\overline{p}$ follows after 
integrating \eqref{class-2} with respect to $\overline{x}$ and using the boundary condition $(\ref{dim-prob})-(iv)$. More explicitly,  
\begin{equation}\label{p-sol}
\overline{p}(\overline{x},\overline{t}) = 1 - 2\phi \overline{s}'(\overline{t})\,\overline{x} - 2Y \int_0^{\overline{x}} \int_{y}^{\overline{s}(\overline{t})} \exp\bigl(\overline{t} - \overline{s}^{-1}(z)\bigr) \,dz\,dy.
\end{equation}

By using classical arguments of partial differential equations we can state that the triple $\{\overline{s}, \overline{\Gamma},\overline{p}\}$ given by \eqref{s-sol},  \eqref{G-sol} and \eqref{p-sol} respectively, where $\overline{s}^{-1}$ is defined through the strict monotonicity of $\overline{s}$,  is after chose direction the unique solution to problem \eqref{dim-prob}. 
 Indeed, let $\{\overline{s_1}, \overline{\Gamma_1},\overline{p_1}\}$ and $\{\overline{s_2}, \overline{\Gamma_2},\overline{p_2}\}$be two solutions of problem  \eqref{dim-prob}. Proceeding as in \eqref{class-2} to \eqref{class-6} we obtain $J_1(\bar{t})=J_2(\bar{t})$ for every $\bar{t}\geq 0$, from which it follows that $\bar{s_1}(\bar{t})=\bar{s_2}(\bar{t})$ for every $\bar{t}\geq 0$, which is a  non-negative increasing function. Consequently, for each fixed $\overline{x}$  this identity yields $\bar{s}_1^{-1}(\overline{x})=\bar{s}_2^{-1}(\overline{x})$,  and the uniqueness of solution to the initial-boundary value problem \eqref{dim-prob}$-(iii)-$\eqref{IC-G} then gives $\bar{\Gamma}_1=\bar{\Gamma}_2$.  Finally, by the uniqueness of solution to system $(\ref{dim-prob})-(i)$, $(\ref{dim-prob})-(ii)$ and $(\ref{dim-prob})-(iv)$ for given boundary and source data ($\bar{s}$ and $\bar{\Gamma}$), we conclude that $\bar{p}_1=\bar{p}_2$.
\end{proof}

\begin{remark} Note that the physics of hydration in a porous media dictates that
the porosity $\phi$ is strictly   positive. Positivity of $Y$ corresponds to water consumption, while $Y<0$ means water release.
From a mathematical perspective, however, our analysis only requires 
$\phi>0$ and no restriction on $Y$, see below.  
\end{remark}

\begin{remark}
 Note also that, at early time $\bar{t}<<1$, the eq.(\ref{s-sol}) provides a solution where $\bar s$ moves as the square root of time, matching the standard infiltration solution in the absence hydration \(Y=0\) given in eq.(\ref{Y=0}). We will quickly implement this assertion in the next section.
\end{remark}
\section{Crossover behavior of the free boundary }

In order to confirm the dual limiting behavior previously observed in  \cite{detournay2025coupled, voller2025assessment},
 whereby the solution approaches two distinct limit functions in the early  and late time regimes, we examine the asymptotic behavior of the exact solution \eqref{s-sol} in both limits.
\begin{theo}
Let us suppose that  $(\overline{s}, \overline{p}, \overline{\Gamma})$ is a unique solution to \eqref{dim-prob}.\\
(i) If $\phi>0$, then
$$
\lim_{\overline t\to 0^+} \overline s(\overline t)/ \sqrt{\frac {\overline t}\phi }=1.
$$
(ii) If $\phi$ and $\phi+Y$ are positive, then
$$
\lim_{\overline t\to \infty} \overline s(\overline t)/ 
\sqrt{\frac {\overline t}{\phi+ Y }}=1.
$$
\end{theo}
\begin{proof} If $\phi+Y=0$, then (i) follows immediately from formula \eqref{Y-PHI=0}. Otherwise, for early times, we use a Taylor expansion for the exponential term  
\begin{equation}\label{taylor}
\exp\left(-\left(1+\frac{Y}{\phi}\right)\overline{t}\right) = 1 - \left(1+\frac{Y}{\phi}\right)\overline{t} + O(\overline{t}^2)\end{equation}
and substitute \eqref{taylor} into \eqref{s-sol} to obtain
\[
\overline{s}^2(\overline{t}) = \frac{1}{\phi+Y}\overline{t} + \frac{Y}{(\phi+Y)^2}\left[\left(1+\frac{Y}{\phi}\right)\overline{t} + O(\overline{t}^2)\right].
\]
The linear terms combine as
$
\frac{1}{\phi+Y}\overline{t} + \frac{Y}{(\phi+Y)^2}\left(1+\frac{Y}{\phi}\right)\overline{t} = \frac{1}{\phi}\overline{t}.
$
Therefore,
\[
\overline{s}(\overline{t}) \sim \sqrt{\frac{\overline{t}}{\phi}} \quad \text{as} \quad \overline{t}\to 0^+.
\]

Part (ii). When $\phi>0$ and $\phi+Y>0$, then for large
times  the exponential contribution vanishes leading to 
\begin{equation*}
\lim_{\bar t\to \infty} \overline{s}(\overline{t})
\bigg/\sqrt{\frac{\overline{t}}{\phi+Y}} =1.
\qedhere
\end{equation*}
\end{proof}
Interestingly, for small values of $\bar t$ the behavior of the interface is universal, depending only on positive $\phi$. More precisely \(\overline{s} \sim \sqrt\frac{\overline{t}}{\phi}\); which, as noted above,  is the infiltration solution in the absence of hydration (\(Y=0\)) .  We also notice that 
the free boundary exhibits a clear crossover from an early-time behavior to a late-time behavior, which depends on $Y$. When $Y+\phi$ is positive, then we see
\(\overline{s} \sim \sqrt{\frac{\overline{t}}{(\phi + Y)}}\). 

Figure \ref{Figure1} illustrates the possible behavior of analytical solution given by eq.(\ref{s-sol})  for  the movement of an infiltration front $\bar{s}(t)$  in the presence of a hydration reaction. 
Figure \ref{fig:1a} shows a log-log plot of  the advance  of the infiltration front for cases where  $Y>-\phi$ and $\phi>0$. 
This plot highlights the early to late time square-root crossover and, for physical hydration cases ($Y>0$), also analytically validates  the previously presented numerical solution in \cite{detournay2025coupled, voller2025assessment}. 
Figure \ref{fig:1b} shows a log-log plot for the movement of the infiltration front  for a special case where $Y+\phi=0$ (linear late time)  and the more general case where $Y<-\phi$ (exponential late time)




\begin{figure}[ht]
    \centering
    \begin{subfigure}[b]{0.48\linewidth}
        \centering
        \includegraphics[width=\linewidth]{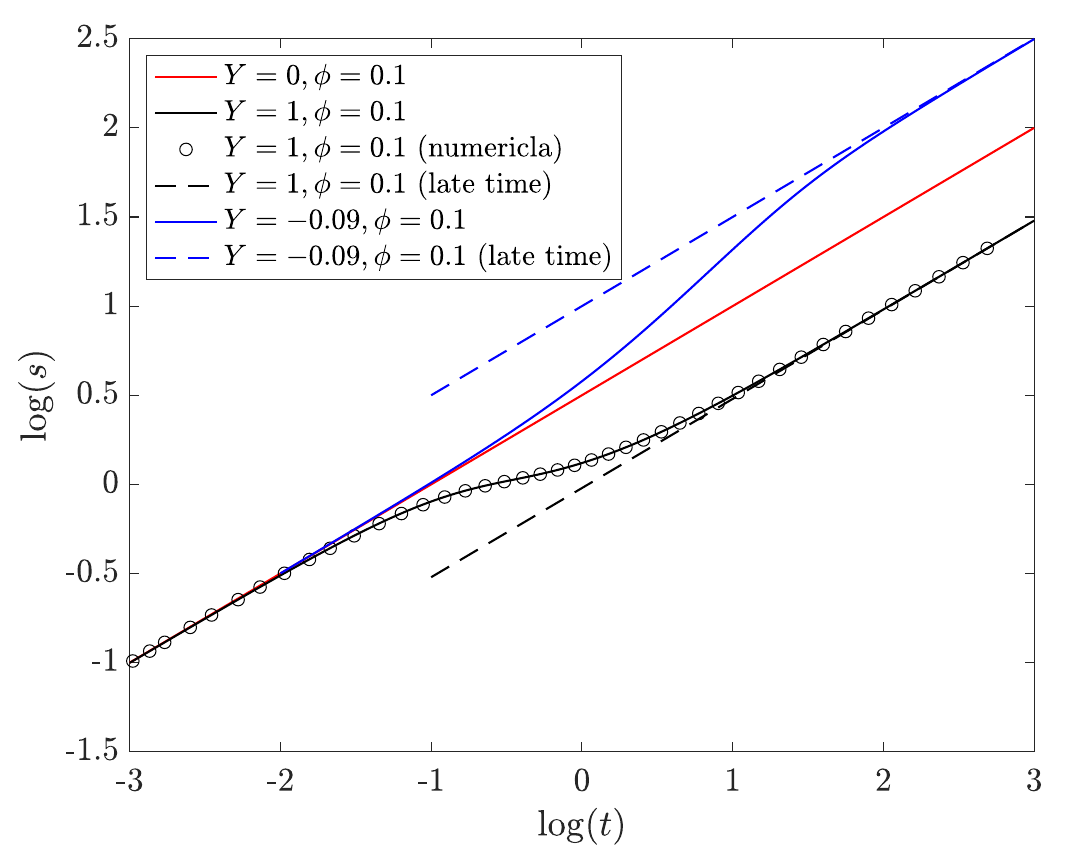}
        \caption{{ Infiltration front movements when $\phi>0$ and $Y+\phi>0$, highlighting the early and late time square-root-in-time behaviors; dashed lines in figure have a slope of $\frac 1 2$.  In the special case where $Y=0$ (red line) the front follows a single square-root for all $t>0$.  When $0>Y>-\phi$  (blue line)  the late time square-root  has a larger pre-factor,  whereas for Y>0 the pre-factor is smaller. The black circles in the figure are obtained using the $Y>0$ numerical solution presented in  \cite{detournay2025coupled, voller2025assessment}
        }}
        \label{fig:1a}
    \end{subfigure}
    \hfill
    \begin{subfigure}[b]{0.48\linewidth}
        \centering
        \includegraphics[width=\linewidth]{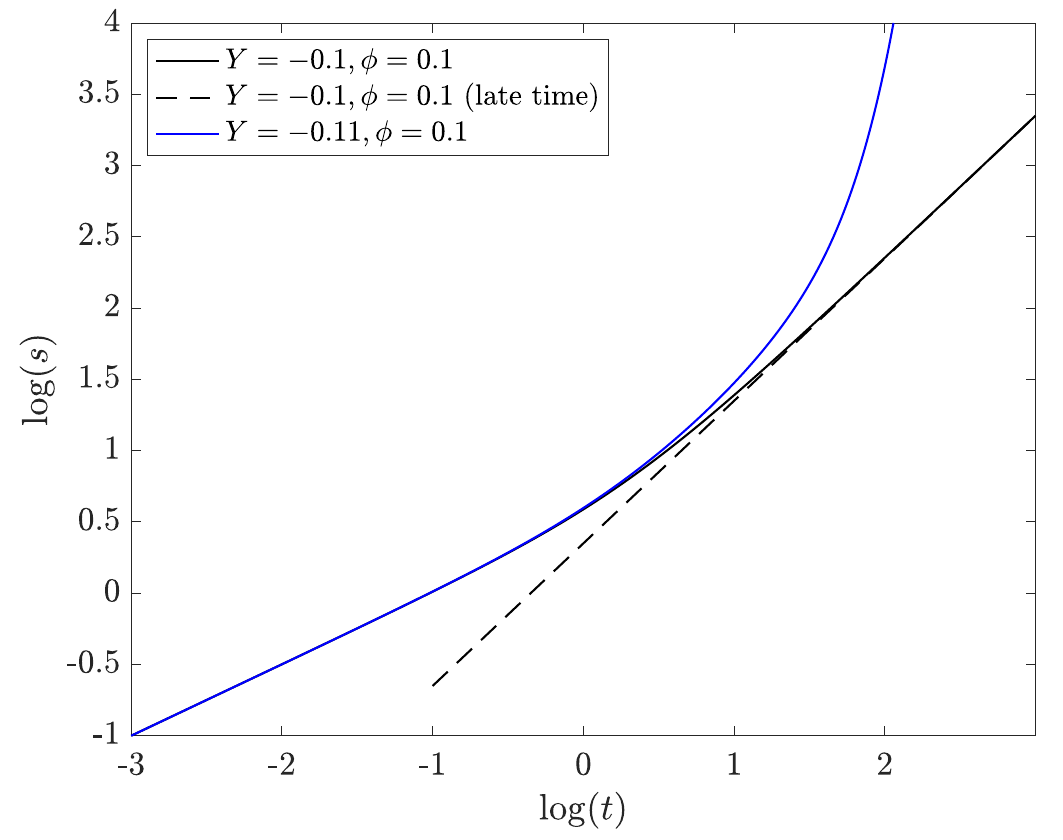}
        \caption{{ Infiltration front movements when $\phi>0$ and $Y\leq -\phi$. In all cases the early time solution follows the square-root-in-time solution associated with  no hydration $(Y=0)$.
        In the special case where $Y=-\phi$ (full black line), as time increases the trajectory of the front crosses over to a linear in time behavior; the dashed black line  has a slope of 1. In the more general case where $Y<-\phi$, (blue line)  the later time behavior is exponential.        
        }}
        \label{fig:1b}
    \end{subfigure}
    \caption{Behavior of the analytical solution \eqref{s-sol} for the movement of the infiltration front $s(t)$ in the presence of a hydration reaction. (a) Case $\phi>0$, $Y>-\phi$. (b) Case $\phi>0$, $Y\leq-\phi$.}
    \label{Figure1}
\end{figure}

\section{Conclusions}
In the absence of a hydration reaction, the problem introduced here reduces to the form of a  Stefan melting problem in the limit of a vanishing specific heat. This is perhaps the most basic of free  boundary problems with a  square-root-in-time dependence  for the free boundary. The introduction of the hydration reaction, however, greatly enriches this problem. Providing, as demonstrated here, a closed expression for the free boundary.
{ At early  times this solution matches the square-root-in-time dependence of infiltration in the absence of hydration, giving rise to a pre-factor for the square-root-in-time that only depends on the porosity of the medium.  The reason for this behavior is that
at early time the front is advancing at a rate which is significantly faster than the hydration reaction can consume water. At later time, however, depending on the consumption $Y$ of the hydration reaction, different behaviors for the front advance are possible; (i) a  square-root-in-time behavior, with a pre-factor that depends on both the porosity and and consumption ($Y>-\phi$), (ii)  a linear in time behavior depending solely on porosity ($Y=-\phi$), or (iii) an exponential behavior ($Y<-\phi$). } 

Beyond its interesting mathematical nature, the closed solution of the infiltration-hydration problem also has practical consequences. In particular, it can be used as the basis of an analysis to investigate the mechanics of the reaction fracture that may occur due to the in place eigenstrains induced by the hydration swelling of the porous media \cite{detournay2026damkohler}.

\section*{Acknowledgment}
The work of DG was in part performed during his visit to the University of Warsaw, which was supported by the excellence center IDUB. SR and DG were partially supported by the project PIP N° 11220220100532 from CONICET, and SR by Proyecto Austral N°006-25CI2001, from Universidad Austral. VV was supported as part of the Center on Geo-processes in Mineral Carbon Storage, an Energy Frontier Research Center funded by the U.S. Department of Energy, Office of Science, Basic Energy Sciences, under Award DE-SC0023429.{   The authors are grateful for discussions with Professor Emmanuel Detournay at the University of Minnesota}





\end{document}